\documentclass{amsart}
\usepackage{graphicx}
\usepackage[all]{xy}
\usepackage[ansinew]{inputenc}
\usepackage{amsmath}
\usepackage{amscd}
\usepackage{amssymb}
\usepackage{latexsym}
\usepackage[active]{srcltx}
\usepackage{mathrsfs}
\usepackage{color}
\usepackage{enumerate}
\usepackage{hyperref}

\newtheorem{thm}{Theorem}[section]

\theoremstyle{definition}
\newtheorem{cor}[thm]{Corollary}
\newtheorem{lem}[thm]{Lemma}
\newtheorem{prop}[thm]{Proposition}
\newtheorem{defn}[thm]{Definition}

\newtheorem{fact}[thm]{Fact}

\newtheorem*{thmA}{Theorem A}
\newtheorem*{thmB}{Theorem B}
\newtheorem*{thmC}{Theorem C}

\newtheorem*{ques}{Open question}
\numberwithin{equation}{section}

\newcommand{\N}{\mathbb{N}}
\newcommand{\Z}{\mathbb{Z}}
\newcommand{\Q}{\mathbb{Q}}
\newcommand{\R}{\mathbb{R}}

\newcommand{\supp}{\operatorname{supp}}

\def \gcd {\operatorname{gcd}}

\newcommand{\Cal}{\mathcal}

\def \<{\langle}
\def \>{\rangle}

\def \((  {(\!(}
\def \)) {)\!)}

\begin{document}

\title[Generalized Cantor sets]{Definability and decidability in expansions by generalized Cantor sets}

\author[W. Balderrama]{William Balderrama}
\address
{Department of Mathematics\\University of Illinois at Urbana-Champaign\\1409 West Green Street\\Urbana, IL 61801}
\email{balderr2@illinois.edu}

\author[P. Hieronymi]{Philipp Hieronymi}
\address
{Department of Mathematics\\University of Illinois at Urbana-Champaign\\1409 West Green Street\\Urbana, IL 61801}
\email{phierony@illinois.edu}
\urladdr{http://www.math.uiuc.edu/\textasciitilde phierony}



\begin{abstract}
We determine the sets definable in expansions of the ordered real additive group by generalized Cantor sets. Given a natural number $r\geq 3$, we say a set $C$ is a generalized Cantor set in base $r$ if there is a non-empty $K\subseteq\{1,\ldots,r-2\}$ such that $C$ is the set of those numbers in $[0,1]$ that admit a base $r$ expansion omitting the digits in $K$. While it is known that the theory of an expansion of the ordered real additive group by a single generalized Cantor set is decidable, we establish that the theory of an expansion by two generalized Cantor sets in multiplicatively independent bases is undecidable.
\end{abstract}

\thanks{This is a preprint version.  Later versions might still contain significant changes.  Comments
are  welcome! The first author was supported by a Graduate Assistance in Areas of National Need (GAANN) Fellowship. The second author was partially supported by NSF grant DMS-1300402. }

\maketitle

\section{Introduction}

One of the most famous and well-studied objects in mathematics is the \textbf{middle-thirds Cantor set} $C$, a set that is constructed by repeatedly removing middle-thirds from the unit interval. As pointed out by Dolich, Miller, and Steinhorn \cite{DMS}, when we expand $(\R,<)$ by a predicate for $C$, the resulting structure is model-theoretically tame. However, by Fornasiero, Hieronymi, and Miller \cite{FHM}, the expansion $(\R,<,+,\cdot,C)$ of the real field by $C$ defines $\N$ and hence every projective set\footnote{Projective in sense of descriptive set theory. See Kechris \cite[Chapter V]{kechris}.}. This immediately raises the question of what happens when adding $C$ to a structure between $(\R,<)$ and $(\R,<,+,\cdot)$.\newline

\noindent In this note we will consider the expansion of the ordered real additive group $(\R,<,+)$ by $C$. In fact, we will not only consider expansions by the usual middle-thirds Cantor set, but by \textbf{generalized Cantor sets}. Given a natural number $r\geq 3$ and a nonempty $K\subseteq\{1,\ldots,r-2\}$, we define $C_{r,K}$ to be the set of those numbers in $[0,1]$ admitting a base $r$ expansion that omits the digits in $K$. We call $r$ the base of $C_{r,K}$. The classical middle-thirds Cantor set is then just $C_{3,\{1\}}$.\newline


\noindent While it has never been stated explicitly, it is known that the theory of the structure $(\R,<,+,C_{r,K})$ is decidable. For $r\in \N_{\geq 2}$, consider the expansion $\Cal T_r$ of $(\R,<,+)$ by a ternary predicate $V_r(x,u,k)$ that holds if and only if $u$ is an integer power of $r$, $k\in \{0,\dots,r-1\}$, and the digit of some base $r$ representation of $x$ in the position corresponding to $u$ is $k$. As shown in Boigelot, Rassart, and Wolper \cite{BRW}, it follows from B\"uchi's work in \cite{Buchi} that the theory of $\Cal T_r$ is decidable. For every non-empty $K\subseteq\{1,\ldots,r-2\}$, the Cantor set $C_{r,K}$ is $\emptyset$-definable in $\Cal T_r$, and therefore the theory of $(\R,<,+,C_{r,K})$ is decidable.\newline

\noindent This leads to the following two natural questions which we will address:
\begin{itemize}
\item[(Q1)] What can be said about sets definable in $(\R,<,+,C_{r,K})$?
\item[(Q2)] Are expansions of $(\R,<,+)$ by multiple generalized Cantor sets model-theoretically tame?
\end{itemize}
Before we address these, let us fix some notation. Say that two expansions $\Cal R$ and $\Cal R'$ of $(\R,<)$ are \textbf{interdefinable} if they define the same sets (with parameters). In such a situation, we write $\Cal R = \Cal R'$.
Let $W_r$ be the intersection of $V_r$ with $[0,1]\times r^{-\N} \times \{0,\dots,r-1\},$ and set $\Cal S_r := (\R,<,+,W_r)$.

\begin{thmA}\label{thm:1} Let $r\in \N_{\geq 3}$, and let $K\subseteq \{1,\ldots,r-2\}$ be nonempty. Then
\[
(\R,<,+,C_{r,K})=\Cal S_r.
\]
\end{thmA}

\noindent Theorem A determines the definable sets in an expansion by a single generalized Cantor set, giving an answer to our first question. A reader looking for a more detailed description of definable sets in $\Cal S_r$ may want to consult \cite{BRW} where a precise automata-theoretic description of definable sets in $\Cal T_r$ (and hence in $\Cal S_r$) is given. Observe that by Belegradek \cite[Corollary 1.7]{Belegradek}, $\Cal S_r$ does not define $\N$. Therefore, $(\R,<,+,C_{r,K})$ is not interdefinable with $\Cal T_r$.  While the theory of $(\R,<,+,C_{r,K})$ is decidable, it is very easy to deduce from Theorem A that the structure does not satisfy any of the combinatorial tameness notions invented by Shelah, such as NIP, NTP2, or $n$-dependence (see also Hieronymi and Walsberg \cite[Theorem B]{HW-Monadic}).
\newline

\noindent We now turn to the second question about expansions by multiple Cantor sets. Let $r,s\in \N_{\geq 3}$, and $K\subseteq \{1,\ldots,r-2\}$ and $L\subseteq \{1,\ldots,s-2\}$ be non-empty. Observe that whenever $\log_r(s)\in \Q$, we have $(\R,<,+,W_r)=(\R,<,+,W_s)$. This statement follows easily from the fact that $W_r$ and $W_{r^{\ell}}$ can be expressed in terms of each other for $\ell \in \N_{\geq 1}$.
Therefore, Theorem A immediately implies that if $\log_r(s) \in \Q$, then $(\R,<,+,C_{r,K},C_{s,L})= \Cal S_r$. We can thus restrict our attention to the case where $\log_r(s)\notin \Q$. In this situation, we are able to prove the following result.

\begin{thmB} Let $r,s\in \N_{\geq 3}$ with $\log_r(s)\notin\Q$, and let $K\subseteq \{1,\ldots,r-2\}$ and $L\subseteq \{1,\ldots,s-2\}$ be non-empty. Then $(\R,<,+,C_{r,K},C_{s,L})$ defines every compact set.
\end{thmB}

\noindent The theory of an expansion that defines every compact set is clearly undecidable, as it defines an isomorphic copy of $(\R,+,\cdot,\N)$. Indeed, every projective subset of $[0,1]^k$ is definable in such an expansion. However, while $(\R,<,+,C_{r,K},C_{s,L})$ defines every compact set, multiplication on $\R$ does not need to be definable, by Pillay, Scowcroft, and Steinhorn \cite{pss-betweengroupsandrings}.
\newline

\noindent We deduce Theorem B directly from Theorem A and the following analogue of Villemaire's theorem \cite[Theorem 4.1]{Villemaire}.

\begin{thmC} Let $r,s\in \N_{\geq 2}$ be such that $\log_r(s) \notin \Q$. Then $(\R,<,+,W_r,W_s)$ defines every compact set.
\end{thmC}

\noindent A few remarks about the proof of Theorem C are in order. In the case where $r$ and $s$ are relatively prime, Theorem C follows from a slight generalization of Hieronymi and Tychonievich \cite[Theorem A]{HT} without significant use of further technology. However, when $r$ and $s$ share a common prime factor, we need to rely in addition on earlier ideas from \cite{Villemaire}. This extra complication arises from the fact that whenever $r$ and $s$ share a common prime factor, the set of numbers admitting a finite base $r$ expansion intersects non-trivially with the set of numbers admitting a finite base $s$ expansion.\newline

\noindent It is natural to ask whether there are any interesting structures between $(\R,<,+)$ and $(\R,<,+,\cdot)$ such that the theory of the expansion of such a structure by a single generalized Cantor set remains decidable. However, the answer to such a question is probably negative. For example, fix $a \in \R$ and let $\lambda_a \colon \R \to \R$ be the function that maps $x$ to $ax$, and consider $(\R,<,+,\lambda_a,C_{r,K})$ for some generalized Cantor set $C_{r,K}$. It was already pointed out in Fornasiero, Hieronymi, and Walsberg \cite[Corollary 3.10]{FHW} that $(\R,<,+,\lambda_a,C_{3,1})$ defines every compact set whenever $a$ is irrational. An inspection of the proof shows that the same argument works for a generalized Cantor set $C_{r,K}$.\newline

\noindent We finish with a remark about the optimality of Theorem A. Observe that it is an immediate consequence of Theorems A and C that $C_{r,K}$ is not definable in $\Cal S_s$ whenever $\log_r(s) \notin \Q$. This consequence is a very special case of a version of Cobham's Theorem for such expansions due to Boigelot, Brusten, and Bruy{\`e}re \cite{BBB2010}. Indeed, if $\log_r(s)\notin \Q$ and $X \subseteq \R$ is both definable in $\Cal S_r$ and weakly recognizable\footnote{See \cite{BBB2010} for a precise definition of weakly recognizable. By Maler and Staiger \cite{MS97} and \cite[Lemma 2.5]{BBB2010}, a subset $X\subseteq \R$ definable in $\Cal S_r$ is weakly recognizable if and only if $X$ is both $F_{\sigma}$ and $G_{\delta}$.}, then $X$ is definable in $\Cal S_s$ if and only if $X$ is definable in $(\R,<,+,\Z)$. See Charlier, Leroy, and Rigo \cite{CLR} for an interesting restatement of this result in terms of graph directed iterated function systems. This suggests that it would be natural to expect that Theorem A holds for a larger class of definable sets in $\Cal S_r$. The obvious extension to weakly recognizable sets fails. To see this, observe that $r^{-\N}$ is definable in $\Cal S_r$ and weakly recognizable, but every subset of $\R$ definable in $(\R,<,+,r^{-\N})$ either has interior or is nowhere dense. The latter statement follows easily from Friedman and Miller \cite[Theorem A]{FM-Sparse} (see \cite[Theorem 7.3]{FHW}). Since $\Cal S_r$ defines sets that are both dense and codense in $(0,1)$, it follows that $W_r$ can not be definable in $(\R,<,+,r^{-\N})$. Nevertheless, we can imagine that Theorem A extends to sets that share the same topological properties as the generalized Cantor sets. As we do not see how our proof generalizes to this setting, we leave this as an open question.

\begin{ques}
Let $r\in \N_{\geq 2}$, and let $C\subseteq \R$ be a nonempty compact set $\emptyset$-definable in $\Cal S_r$ that has neither interior nor isolated points. Is $(\R,<,+,C)=\Cal S_r$?
\end{ques}

\noindent This question has a negative answer\footnote{We thank Erik Walsberg for pointing this out.} when parameters can be used to define $C$. In \cite[Section 7.2]{FHW} a subset $E_S\subseteq \R$ is constructed such that $E_S$ is compact, neither has interior nor isolated points, and $(\R,<,+,E_S)$ does not define a dense and codense subset of $(0,1)$. It is clear from the construction of $E_S$ that $E_S$ is definable in $\Cal S_2$.

\subsection*{Acknowledgements} The authors thank Alexis B\`es, Bernarnd Boigelot, V\'eronique Bruy\`ere, Christian Michaux, and Fran\c{c}oise Point for answering their questions and pointing out references, and Erik Walsberg for helpful comments.

\subsection*{Notations} We will now fix a few conventions and notations. First of all, $\N$ denotes the set of natural numbers including $0$. When we say ``\emph{definable}'', we mean ``\emph{definable possibly with parameters}''.
\noindent Let $r \in \N_{\geq 2}$ and let $\Sigma_r = \{0,\ldots,r-1\}$. Let $x\in \R$. A \textbf{base $r$ expansion} of $x$ is an infinite $\Sigma_r\cup \{\star\}$-word $a_p\cdots a_0\star a_{-1}a_{-2}\cdots$ such that
\begin{equation}\label{eq:1}
z = -a_p r^p + \sum_{i=-\infty}^{p-1} a_{i} r^i
\end{equation}
with $a_p \in \{0,r-1\}$ and $a_{p-1},a_{p-2},\ldots \in \Sigma_r$. We will call the $a_i$'s the \textbf{digits} of the base $r$ expansion of $x$. The digit $a_k$ is \textbf{the digit in the position corresponding to $r^{k}$}. We define $V_r(x,u,k)$ to be the ternary predicate of $\R$ that holds whenever there exists a base $r$ expansion $a_p\cdots a_0\star a_{-1}a_{-2}\cdots$ of $x$ such that $u=r^{n}$ for some $n\in \Z$ and $a_{n} = k$. As is commonly done, we will often identify the word $a_p\cdots a_0\star a_{-1}a_{-2}\cdots$ with the expression in \eqref{eq:1}.\newline

\noindent A number $x\in\R$ can possibly admit two distinct base $r$ expansions. We can use the following to pick out a preferred expansion. Define $U_r(x,u,k)$ to be the ternary predicate of $\R$ that holds whenever $x\in \R$ and there is a base $r$ expansion $a_p\cdots a_0\star a_{-1}a_{-2}\cdots$ of $x$ such that
\begin{itemize}
\item $a_{-i} \neq r-1$ for infinitely many $i \in \N$,
\item $u=r^{-n}$ for some $n\in \Z$, and
\item $a_{n} = k$.
\end{itemize}

\noindent Let $X\subseteq \R$. When we refer to the \textbf{restriction of $V_r$ to $X$}, we actually mean the restriction of $V_r$ to $X\times \R \times \R$. Similarly, the restriction of $U_r$ to $X$ refers to the restriction of $U_r$ to $X \times \R \times \R$.

\begin{fact}\label{fact:vu} Let $\Cal R$ be an expansion of $(\R,<)$ that defines $r^{-\N}$. Let $X\subseteq \R$ be definable in $\Cal R$. Then the following are equivalent:
\begin{enumerate}
\item $\Cal R$ defines the restriction of $V_r$ to $X$;
\item $\Cal R$ defines the restriction of $U_r$ to $X$.
\end{enumerate}
\end{fact}
\begin{proof}
Let $x \in \R$. Then $x$ has at most two base $r$ expansions, and if $x$ has two base $r$ expansions, then there is $n \in \N_{>0}$ and $a_p,\cdots,a_0,a_{-1},\cdots,a_{-n}\in \Sigma_r$ such that
\[
a_p \cdots a_0\star a_{-1}\cdots a_{-n} \hbox{ and } a_p\cdots a_0\star a_{-1} \cdots a_{-(n-1)}(a_{-n}-1)(r-1)(r-1)\cdots
\]
are the two base $r$ expansions of $x$. Thus,
\begin{align*}
V_r(x,u,k)\iff U_r(x,u,k)\vee \Big[\exists v\in r&^{-\N}\neg U_r(x,v,0)\\
&\wedge (\forall t\in r^{-\N}(t<v)\rightarrow U_r(x,v,0))\\
&\wedge (u=v\rightarrow U_r(x,u,k+1))\\
&\wedge ((u<v)\rightarrow (k=r-1))\Big].
\end{align*}
Therefore, (2) implies (1). The other direction is similar.
\end{proof}



\section{Proof of Theorem A}

Let $r\in \N_{\geq 3}$, and let $K\subseteq \{1,\ldots,r-2\}$ be nonempty. In this section, we show that $(\R,<,+,C_{r,K})=\Cal S_r$. For ease of notation, we will write $C$ for $C_{r,K}$ in this section. Since $C$ is definable in $\Cal S_r$, it is only left to show that $W_r$ is definable in $(\R,<,+,C)$. To do this, we first show that the definability of $W_r$ follows from the definability of the restriction of $V_r$ to $C$, and then we show that the restriction of $V_r$ to $C$ is in fact definable. Throughout the rest of this section, \emph{``definable''} will mean \emph{``definable in $(\R,<,+,C)$''}. \newline

\noindent Let $k_1,\dots k_l\in K$ and $m_1,\dots,m_l\in \Sigma_{r}\setminus K$ be such that
\begin{itemize}
\item $k_1<m_1<k_2<m_2<\dots<k_l<m_l$ and
\item $K = \N \cap \bigcup_{i=1}^l [k_i,m_i)$.
\end{itemize}
Set $M := \{ m_1,\ldots,m_l\}$.\newline

\noindent Recall that $C$ is the set of elements in $[0,1]$ that admit a base $r$ representation omitting the digits in $K$, and that $r-1\notin K$. Therefore, for every subset $X\subseteq \N_{>0}$, there is some $c \in C$ whose base $r$ representation is
\[
c = \sum_{n \in X} (r-1)r^{-n}.
\]
From this observation, we deduce directly that for every $x\in [0,1]$ there are $c_1,\dots,c_{r-1} \in C$ such that
\[
x = \frac{1}{r-1}(c_1 + \dots + c_{r-1}).
\]
This is an analogue of the standard fact that Minkowski sum of the middle-thirds Cantor set with itself is the interval $[0,2]$. Define $E\subseteq C^{r-1}$ to be the set of all tuples $(c_1,\dots,c_{r-1})$ such that each $c_i$ admits a base $r$ expansion in which only the digits $0$ and $r-1$ occur. Let $h\colon E\times r^{-\N_{>0}}\to \{0,r-1\}$ be the function that maps the tuple $(c_1,\dots,c_{r-1},r^{-n})$ to the cardinality of the set $\{ i \in \{0,\dots,r-1\} \ : \ V_r(c_i,r^{-n},r-1)\}$. Note that both $E$ and $h$ are definable if the restriction $V_r$ to $C$ is definable.

\begin{lem}\label{lem:cexpansion} Let $x\in [0,1)$, $n\in \N_{\geq 1}
,$ and $k\in \{0,\dots,r-1\}$. Then $V_r(x,r^{-n},k)$ holds if and only if
there is $c=(c_1,\dots,c_{r-1})\in E$ such that
\begin{itemize}
\item $x = \frac{1}{r-1}(c_1 + \dots + c_{r-1})$, and
\item $h(c,r^{-n})=k.$
\end{itemize}
\end{lem}
\begin{proof} Suppose that $x= \sum_{i=1}^{\infty} a_i r^{-i}$ is some base $r$ expansion of $x$. For $j\in\{1,\dots,r-1\},$ set
\[
c_j := \sum_{\substack{i \in \N_{>0},\\ a_i\geq j}} (r-1) r^{-i}.
\]
Then $(c_1,\dots,c_{r-1}) \in E$, and $x = \frac{1}{r-1}(c_1 + \dots + c_{r-1})$. Moreover, $h(c_1,\dots,c_{r-1},r^{-n})=a_n$.\newline

\noindent Suppose next that there is $c=(c_1,\dots,c_{r-1})\in E$ such that $x = \frac{1}{r-1}(c_1 + \dots + c_{r-1})$. Then we get the following base $r$ expansion of $x$:
\begin{align*}
x &= \frac{1}{r-1}\sum_{i=1}^{r-1}c_i\\
&= \frac{1}{r-1} \sum_{i=1}^{r-1} \sum_{\substack{i \in \N_{>0},\\ V_r(c_i,r^{-i},r-1)}} (r-1) r^{-i}\\
&=  \sum_{i \in \N_{>0}} h(c,r^{-i}) r^{-i}.
\end{align*}
Thus, $V(x,r^{-n},h(c,r^{-n}))$.
\end{proof}
\noindent By Lemma \ref{lem:cexpansion}, the definability of $W_r$ follows from the definability of the restriction of $V_r$ to $C$. To establish the definability of the restriction of $V_r$ to $C$, we will rely heavily on the regularity of the complementary intervals of $C$.

\begin{defn}A \textbf{complementary interval} of $C$ is an open interval $\big(c_1,c_2\big)\subseteq [0,1]$ such that $c_1,c_2 \in C$ but $\big(c_1,c_2\big) \cap C = \emptyset$.
\end{defn}

\noindent  For example, $(\tfrac19,\tfrac29)$ is a complementary interval of the middle-thirds Cantor set. For $a \in \R_{>0}$, denote by $R_a$ the set of right endpoints of complementary intervals of $C$ that are of length at least $a$. Observe that the set $R:=\{ (a,x) : x \in R_a\}$ is definable. Let $D$ be the set of all right endpoints of $C$. As $D=\bigcup_{a \in \R_{>0}} R_a$, $D$ is definable. Moreover, the set
\[
L :=\{ z \in \R \ : \ z \hbox{ is the length of a complementary interval of $C$ in $[0,1]$} \}
\]
is definable.

\begin{lem}\label{lem:R} Let $d\in (0,1]$ and let $n \in \N$. Then, the following are equivalent:
\begin{enumerate}
\item $d \in R_{r^{-n}}$;
\item there are $b_1,\dots,b_{n-1} \in \Sigma_r \setminus K$ and $b_n \in M$ such that $d = \sum_{i=1}^n b_i r^{-i}$.
\end{enumerate}
\end{lem}
\begin{proof}
Suppose first that $d = \sum_{i=1}^n b_i r^{-i}$ with $b_i \in\Sigma_r\setminus K$ for $i<n$ and $b_n \in M$. Suppose towards a contradiction that there is $c \in C$ such that $0<d-c<r^{-n}$. We can assume without loss of generality that $c$ has a unique base $r$ expansion $\sum_{i=1}^{\infty} b_i' r^{-i}$. Our assumption that $0<d-c<r^{-n}$ implies $c-r^{-n}<d<c$, or in other words,
\[
 (b_{n}-1)r^{-n} + \sum_{i=1}^{n-1} b_i r^{-i} < \sum_{i=1}^\infty b_i'r^{-i} <  \sum_{i=1}^n b_i r^{-i}.
\]
Thus $b_i'=b_i$ for $i<n$, and $b_n'=b_{n}-1$. Since $b_n \in M$, $b_n'\in K$. Since $c$ has only one base $r$ expansion, $c\notin C$. This is a contradiction.\newline
Suppose next that $d\in R_{r^{-n}}$. Because $d\in C$, we can write $d$ as $\sum_{i=1}^{\infty} b_i r^{-i}$ with each $b_i \in\Sigma_r \setminus K$. Then the truncation $d_n:=\sum_{i=1}^{n} b_i r^{-i}$ is in $C$, with $0 \leq d - d_n \leq r^{-n}$. Since $d$ is the right endpoint of a complementary interval of length $r^{-n}$, it follows that either $d=d_n$ or $d=d_n+r^{-n}$. But it cannot be the latter, for if $d=d_n+r^{-n}$, then $c=d_n+(r-1)d^{-(n+1)}\in C$ with $0<d-c<r^{-n}$, contradicting the assumption that $d\in R_{r^{-n}}$. Thus, $d=d_n$.
It is left to show that $b_n \in M$. Suppose towards a contradiction that $b_n-1 \notin K$. Then
\[
c'=(b_n-1)r^{-n} + (r-1)r^{-(n+1)} + \sum_{i=1}^{n-1} b_i r^{-i} \in C,
\]
and $d-c'<r^{-n}$, again contradicting the assumption that $d\in R_{r^{-n}}$. Thus, $b_n \in M$, and $d$ has the desired form..
\end{proof}

\begin{cor}\label{cor:length} Let $d\in D$, and let $n\in \N_{>0}$, $b_1,\dots,b_{n-1} \in \Sigma_r \setminus K$, $m_j \in M$ be such that $d = \sum_{i=1}^{n-1} b_i r^{-i}+m_jr^{-n}$.  Then the length of the complementary interval with right endpoint $d$ is $(m_j-k_j)r^{-n}$.
\end{cor}
\begin{proof}
It can be checked easily that the complementary interval with right endpoint $d$ is exactly the interval
\[
\left(\left(\sum_{i=1}^{n-1} b_i r^{-i}\right)+ k_j r^{-n},\left(\sum_{i=1}^{n-1} b_i r^{-i}\right) + m_j r^{-n}\right).
\]
The length of this interval is $(m_j-k_j)r^{-n}$.
\end{proof}

\noindent The following description of $L$ follows immediately from Corollary \ref{cor:length}.

\begin{cor}\label{cor:L} The set $L$ is equal to
\[
\{ (m_i-k_i) r^{-n} \ : \ i \in \{1,\dots,l\}, n \in \N_{>0}\} = \bigcup_{i=1}^l (m_i-k_i) r^{-\N_{> 0}}.
\]
\qed
\end{cor}

\begin{cor}\label{cor:rndef} The set $r^{-\N}$ is definable.
\end{cor}
\begin{proof}
Define $v \in \Sigma_r$ by
\[
v := \min_{i\in \{1,\dots,l\}} (m_i - k_i).
\]
Let $j \in \{1,\dots,l\}$ be minimal such that $m_j-k_j=v$. Let $f\colon D \to L$ be the function that maps $d\in D$ to the length of the complementary interval with right endpoint $d$. Let $D'$ be the set of all $d\in D$ such that there is no $e\in D$ with $e<d$ and $f(e) \leq f(d)$. Observe that both $f$ and $D'$ are definable. It follows from Lemma \ref{lem:R} and Corollary \ref{cor:length} that
\[
D' = \{ m_j r^{-n} \ : \ n \in \N_{>0}\} = m_j r^{-\N_{>0}}.
\]
The definability of $r^{-\N}$ follows.
\end{proof}

\noindent We now use the definability of $r^{-\N}$ to prove the definability of the restriction of $W_r$ to $C$.

\begin{defn} Let $\mu \colon r^{-\N}\times C\to C$ map $(s,c)$ to $\max (R_s \cap (-\infty,c])$ if this maximum exists, and to $0$ otherwise.
\end{defn}

\noindent Observe that $\mu$ is definable, as both $R$ and $r^{-\N}$ are. Loosely speaking, $\mu(r^{-n},c)$ is the best approximation of $c$ from the left by a right endpoint of a complementary interval of length at most $r^{-n}$.
We now establish the precise connection between the function $\mu$ and the base $r$ expansion of elements of $C$.

\begin{lem}\label{lem:descmu} Let $n\in \N$, and let $c = \sum_{i=1}^{\infty} b_i r^{-i}$ be such that $b_i \in \Sigma_r\setminus K$. Then
\[
\mu(r^{-n},c) = \sum_{i=1}^{n-1} b_i r^{-i} + \max \left(M \cap (-\infty, b_n]\right)  r^{-n}.
\]
\end{lem}
\begin{proof}
Set $d:= \sum_{i=1}^{n-1} b_i r^{-i} + \max \left(M \cap (-\infty, b_n]\right)  r^{-n}$. By Lemma \ref{lem:R}, $d\in R_{r^{-n}}$. It is left to show that $(d,c)\cap R_{r^{-n}}$ is empty.
Suppose towards a contradiction that there is $e\in (d,c)\cap R_{r^{-n}}$. Then
$0<c-e < c-d < r^{-(n-1)}$, so by Lemma \ref{lem:R}, there exists $a\in M$ with
\[
e = \sum_{i=1}^{n-1} b_i r^{-i} + a r^{-n}.
\]
Thus $\max \left(M \cap (-\infty, b_n]\right) < a \leq b_n$, and therefore $a \notin M$.
This is a contradiction.
\end{proof}

\noindent In the following, we will show that given an element $c \in C$, we just need to know $\mu(r^{-n},c)$ and $\mu(r^{-(n-1)},c)$ in order to recover the digit in the position corresponding to $r^{-n}$ in a preferred base $r$ expansion of $c$. We now define a set $Z\subseteq \R^3$ that formalizes this idea.

\begin{defn}
Define $Z \subseteq \R^3$ to be the set of all triples $(c,s,d)$ such that $c \in C$, $s \in r^{-\N_{>0}}$, and
\begin{align*}
\bigvee_{i=0}^{r-1} \bigvee_{j=0}^{r-1}& \Big( d=j \ \wedge \ \mu(rs,c) + i rs \leq \mu(s,c)<\mu(rs,c)+(i+1)rs\\
& \wedge \ \mu(rs,c) + irs + js \leq c < \mu(rs,c)+irs + (j+1) s\Big).
\end{align*}
\end{defn}

\begin{lem}\label{lem:zu} The set $Z$ is equal to $U_r \cap \left( C \times \R^2 \right)$.
\end{lem}
\begin{proof}
Let $c\in C$ be such that $c=\sum_{i=1}^{\infty} b_i r^{-i}$, where each $b_i \in \Sigma_r \setminus K$, and $b_i\neq r-1$ for infinitely many $i$.
Let $n \in \N$, and set $s=r^{-(n+1)}$. By Lemma \ref{lem:descmu},
\[
\mu(s,c) - \mu(rs,c) = \big(b_n - \max \left(M \cap (-\infty, b_n]\right)\big)  rs + \max \left(M \cap (-\infty, b_{n+1}]\right)s.
\]
Set $i:= b_n - \max \left(M \cap (-\infty, b_n]\right)$. It follows that
\[
\mu(rs,c) + i rs \leq \mu(s,c)< \mu(rs,c) + (i+1)rs.
\]
Thus, there is $j\in \Sigma_r$ such that
\begin{equation}\label{eq:lemuonc}
\mu(rs,c) + irs + js \leq c < \mu(rs,c)+irs + (j+1) s. \tag{*}
\end{equation}
By Lemma \ref{lem:descmu},
\begin{align*}
c &- (\mu(rs,c) + irs + js) \\
&= c- \left(\left(\sum_{i=1}^{n-1} b_i r^{-i}\right) + \max \left(M \cap (-\infty, b_n]\right)  r^{-n} + i r^{-n} + js\right)\\
&= \sum_{i=n+1}^{\infty} b_i r^{-i} - js.
\end{align*}
From \eqref{eq:lemuonc}, we deduce
\[
0\leq\left(\sum_{i=n+1}^{\infty} b_i r^{-i}\right) - js < s,
\]
or in other words,
\[
0\leq (b_{n+1}-j)r^{-(n+1)}+\sum_{i=n+2}^\infty b_ir^{-i}<r^{-(n+1)}.
\]
Thus $(b_{n+1}-j)r^{-(n+1)}<r^{-(n+1)}$, so that $b_{n+1} =j$. This demonstrates that $U_r(c,s,d)$ if and only if there are $i,j\in \Sigma_r$ such that $d=j$ and $i,j$ satisfy \eqref{eq:lemuonc}. The latter statement holds if and only if $(c,s,d)\in Z$.
\end{proof}

\noindent We can now finish the proof of Theorem A.

\begin{proof}[Proof of Theorem A]
By Lemma \ref{lem:zu}, the restriction of $U_r$ to $C$ is definable. Since $r^{-\N}$ is definable by Corollary \ref{cor:rndef}, the restriction of $V_r$ to $C$ is definable by Fact \ref{fact:vu}. The definability of $W_r$ then follows from Lemma \ref{lem:cexpansion}.
\end{proof}

\section{Finite base $r$ expansions and $\omega$-orderable sets}

Throughout this section, fix some $r\in\N_{\geq 2}$. The purpose of this section is to collect some basic facts we will need about numbers with finite base $r$ expansions. Define $D_r$ to be the set of numbers in $[0,1)$ admitting a finite base $r$ expansion. Notice that $D_r$ is a dense subset of $[0,1)$, and that that $D_r$ is definable in $(\R,<,W_r)$ by
\[
d\in D_r\iff d\in [0,1)\wedge(\exists v>0)(\forall u<v)W_r(d,u,0).
\]
We let $D_1 :=\{0\}$.
Define $\tau_r\colon D_r\rightarrow r^{-\N_{>0}}$ so that $\tau_r(d)$ is the least $u\in r^{-\mathbb{N}_{>0}}$ appearing with nonzero coefficient in the finite base $r$ expansion of $d$. Note that for $x \in D_r$ and $d\in \N_{>0}$, we have $\tau_r(x) = r^{-d}$ if and only if there is $k \in \{0,\dots,r^d-1\}$ such that $x = kr^{-d}$. For $d,e\in D_r$, let
\[
d\prec_r e\iff \tau_r(d)>\tau_r(e)\mbox{ or }\left(\tau_r(d)=\tau_r(e)\mbox{ and }d<e\right).
\]
It is worth distinguishing the following observations.

\begin{lem}\label{lem:omega} The ordered set $(D_r,\prec_r)$ has order type $\omega$.
\end{lem}
\begin{proof}
As $D_r$ is bounded, $\tau^{-1}(r^{-d})$ is finite for each $d\in\N_{>0}$. As $(r^{-\N_{>0}},>)$ has order type $\omega$, the lemma follows.
\end{proof}

\begin{lem}\label{lem:exp}
Let $r=p_1^{\alpha_1}\cdots p_n^{\alpha_n}$ be the prime factorization of $r$, and let $w\in [0,1)$. Then
\begin{itemize}
\item[(1)] $w\in D_r$  if and only if it can be written in the form $$w=\frac{k}{p_1^{k_1}\cdots p_n^{k_n}}$$ where $k,k_1,\ldots,k_n\in\N$.
\item[(2)] if $w\in D_r$ and $d \in \N_{>0}$, then $\tau_r(w)=r^{-d}$ if and only if $d$ is minimal in $\N$ such that $wr^{d} \in \N$.
\end{itemize}
\end{lem}
\begin{proof}
We will prove (1) and leave the easy proof of (2) to the reader. If $w\in D_r$, we can write $w$ as $w_{-1}r^{-1}+\cdots+w_{-l}r^{-l}$ with $0\leq w_i\leq r-1$ for $-l\leq i \leq -1$. Thus $w\cdot r^l\in\N$, and so $w$ is of the desired form. Conversely, if we write
\[
w=\frac{k}{p_1^{k_1}\cdots p_n^{k_n}}=\frac{kp_1^{l-k_1}\cdots p_n^{l-k_n}}{r^l}
\]
with $l\geq\max\{k_1,\ldots,k_n\}$, then $r^{l}w=kp_1^{l-k_1}\cdots p_n^{l-k_n}\in \N$. Hence $r^lw$ has finite base $r$ expansion, and so too does $w$.
\end{proof}

\begin{lem}\label{lem:tau} Let $r=p_1^{\alpha_1}\cdots p_n^{\alpha_n}$ be the prime factorization of $r$, $w\in D_r$, and $m\in\N$ and $d_1,\ldots,d_n\in\Z$ be such that $w=mp_1^{d_1}\cdots p_n^{d_n}$ with $p_i\nmid m$ for $i\in\{1,\ldots,n\}$. Then $\tau_r(w)=r^e$, where
\[
e=\min\left\{\left\lfloor\frac{d_1}{\alpha_1}\right\rfloor,\ldots,\left\lfloor\frac{d_n}{\alpha_n}\right\rfloor\right\}.
\]
\end{lem}
\begin{proof}
Let $e\in \Z$ be maximal such that $d_i-e\alpha_i\geq 0$ for each $i$. Then
\[
r^{-e}w = mp_1^{d_1-e\alpha_1}\cdots p_n^{d_n-e\alpha_n}\in\N.
\]
Since $r\nmid mp_1^{d_1-e\alpha_1}\cdots p_n^{d_n-e\alpha_n}$, $-e$ is the minimal element of $\N$ with this property. By Lemma \ref{lem:exp}(2), $\tau_r(w)=r^{e}$. The statement of the Lemma follows.
\end{proof}

\begin{lem}\label{lem:bool} Let $r,s\in\N_{\geq 2}$. Then
\begin{enumerate}
\item $D_r\cap D_s=D_{\gcd(r,s)}$;
\item if $r$ and $s$ are coprime, then $D_r\cap D_s=\{0\}$;
\item if $r$ and $s$ share the same prime factors, then $D_r=D_s$;
\end{enumerate}
\end{lem}
\begin{proof}
Statement (3) follows directly from Lemma \ref{lem:exp}(1), and Statement (2) is a special case of Statement (1). Therefore, we just need to prove (1). Write $r=p_1^{\alpha_1}\cdots p_n^{\alpha_n}$ and $s=q_1^{\beta_1}\cdots q_m^{\beta_m}$ for the prime factorizations of $r$ and $s$. If $w\in D_r\cap D_s$, then by Lemma \ref{lem:exp}(1) we can write $w$ as
\[
w=\frac{k}{p_1^{k_1}\cdots p_n^{k_n}}=\frac{l}{q_1^{l_1}\cdots q_m^{l_m}}
\] with $k,k_1,\ldots,k_n,l,l_1,\ldots,l_m\in\N$. If these are written in reduced form, then $\{p_i:k_i\neq 0\}=\{q_i:l_i\neq 0\}$ by uniqueness of such presentations. Thus $w\in D_{\gcd(r,s)}$ by Lemma \ref{lem:exp}(1). The other inclusion is immediate by Lemma \ref{lem:exp}(1).
\end{proof}

\noindent Statement (3) of Lemma \ref{lem:bool} was already recognized in \cite[Proof of Theorem 5.3]{BBB2010} as an obstruction to establishing stronger analogues of Cobham's theorem. We will see in the next section that Lemma \ref{lem:bool} is also the reason why the proof of Theorem C is more complicated in the case that $r,s$ are not coprime.

\begin{cor}\label{cor:bool} Let $r,s\in\N_{\geq 2}$ be coprime. Then $(D_r-D_r)\cap (D_s-D_s)=\{0\}$.
\end{cor}
\begin{proof}
Let $a_1,a_2\in D_r$ and $b_1,b_2 \in D_s$ be such that $a_1 - a_2 = b_1 - b_2$. From the definition of $D_r$ and $D_s$ we deduce that  $a_1 - a_2 \in D_r$ and $b_1-b_2\in D_s$ whenever $a_1 -a_2\geq 0$, and that $a_2 -a_1 \in D_r$ and $b_2 - b_1 \in D_s$ whenever $a_1 - a_2<0$. The statement of the corollary follows now directly from Lemma \ref{lem:bool}(2).
\end{proof}

\subsection{Dense $\omega$-orderable sets}
Let $\Cal R$ be an expansion of $(\R,<)$ and $I$ be an interval of $\R$. We say a set $D\subseteq \R$ is a \textbf{dense $\omega$-orderable subset of $I$} in $\Cal R$ if $D$ is dense in $I$ and there exists a definable order $\prec$ on $D$ such that $(D,\prec)$ has order type $\omega$. By Lemma \ref{lem:omega}, $D_r$ is a dense $\omega$-orderable subset of $[0,1)$ in the expansion $(\R,<,W_r)$.\newline

\noindent The following fact is a slight generalization of \cite[Theorem A]{HT} that was first observed in \cite[Proposition 3.8]{FHW}.

\begin{fact}\label{fact:htplus}
Let $\Cal R$ be an expansion of $(\R,<)$. Suppose $\Cal R$ defines an order $(D,\prec)$, an open interval $I\subseteq\R$, and a function $g\colon\R^3\times D\rightarrow D$ such that
\begin{itemize}
\item $(D,\prec)$ has order type $\omega$ and $D$ is dense in $I$, and
\item for every $a,b\in I$ and $e,d\in D$ with $a<b$ and $e\preceq d$,
\[\{c\in\R:g(c,a,b,d)=e\}\cap(a,b)\mbox{ has nonempty interior.}\]
\end{itemize}
Then $\Cal R$ defines every subset of $D^n$ and every open subset of $I^n$ for every $n\in \N$.
\end{fact}

\noindent It is often non-trivial to check whether a given expansion satisfies the assumptions of Fact \ref{fact:htplus}. The next Lemma gives an easy to use criterion when there are multiple dense $\omega$-orderable subsets.

\begin{lem}\label{lem:diff}
Let $\mathcal{R}$ be an expansion of $(\R,<,+)$. If there exist two dense $\omega$-orderable subsets $C$ and $D$ of $(0,1)$ such that $(C-C)\cap (D-D)=\{0\}$, then $\mathcal{R}$ defines every open subset of $(0,1)^n$ for any $n\in\N$.
\end{lem}
\begin{proof}
We essentially follow the proof of \cite[Theorem C]{HT}. Let $\prec_C$ and $\prec_D$ be the definable orders of order type $\omega$ on $C$ and $D$ respectively. Define $h_1\colon\R_{>0}\times C\times C\rightarrow D$ by letting $h_1(u,d,e)$ be the $\prec_D$-minimal $t\in D$ such that $t\in (e,e+u)$ and $t$ is $<$-closer to $e$ than any other element of $C_{\prec_C d}$. Define $h_2\colon\R_{>0}\times C\times C\rightarrow (D-C)$ by $h_2(u,d,e)=h_1(u,d,e)-e$. Notice that for fixed $u\in\R_{>0}$ and $d\in C$, the function $e\mapsto h_2(u,d,e)$ is injective; indeed, if $u\in\R_{>0}$ and $d,e_1,e_2\in C$ are such that $h_2(u,d,e_1)=h_2(u,d,e_2)$, then
\[
h_1(u,d,e_1)-h_1(u,d,e_2)=e_1-e_2\in (C-C)\cap (D-D)=\{0\}.
\]
Thus $e_1=e_2$ as claimed. Define now $g\colon\R^3\times C\rightarrow C$ so that if $a<b$, then $g(c,a,b,d)$ is the $\prec_C$-minimal $e\in C_{\preceq_C d}$ such that $|(c-a)-h_2(b-a,d,e)|$ is minimal. We now claim that Fact \ref{fact:htplus} applies to the ordered set $(C,\prec_C)$ and function $g$. The claim is that for fixed $a<b\in\R$ and $e\preceq_C d\in C$, the set
\[
\{c\in\R:g(c,a,b,d)=e\}\cap (a,b)
\]
has nonempty interior. Notice that $a+h_2(b-a,d,e)\in (a,b)$ and $g(a+h_2(b-a,d,e),a,b,d)=e$. By finiteness of $C_{\preceq_C d}$ and injectivity of $h_2(b-a,d,-)$, there is an open interval $I$ around $a+h_2(b-a,d,e)$ such that for all $c\in I$, $g(c,a,b,d)=e$. This concludes the proof.
\end{proof}

\subsection{Expansions of $\Cal S_r$} We now collect two corollaries of Fact \ref{fact:htplus} when we restrict to the special case that $\Cal R$ is an expansions of $\Cal S_r$.

\begin{prop}\label{prop:infinitepreimage} Let $\ell\in \N_{>0}$ and let $f \colon r^{-\N} \to r^{-\ell\N}$ be such that $f^{-1}(r^{-\ell d})$ is infinite for all $d \in \N$. Then
$(\R,<,+,W_r,f)$ defines every compact set.
\end{prop}
\begin{proof}
Let $Z := \{ (a,b) \in [0,1] \ : \ a < b\}$ and let $\lambda \colon Z \to D_r$ map a pair $(a,b) \in Z$ to the $\prec_r$-minimal element in $(a,b)\cap D_r$.
For $k\in \{1,\dots,\ell\}$, define $h_k \colon Z \times r^{-\ell\N} \to r^{-\N}$ to be the function that maps $(a,b,r^{-\ell d})$ to the $k$-th $<$-largest $r^{-e}\in r^{-\N}$ such that
\begin{itemize}
\item[(1)] $f(r^{-e}) = r^{-d}$,
\item[(2)] $b-\lambda(a,b) > r^{-e+1}$.
\end{itemize}
It follows directly from (1) that if $(a,b) \in Z$, $d,d' \in \N$, and $k,k'\in \{1,\dots,\ell\}$, then
\begin{itemize}
\item[(I)]  $h_k(a,b,r^{-\ell d}) \neq h_{k}(a,b,r^{-\ell d'})$ whenever $d\neq d'$,
\item[(II)] $h_k(a,b,r^{-\ell d}) \neq h_{k'}(a,b,r^{-\ell d})$ whenever $k\neq k'$.
\end{itemize}
For $(a,b) \in Z$, let $Y_{a,b}$ be the set of all $x\in [0,1]$ such that
\[
\forall r^{-e} \in r^{-\N} \left[\bigvee_{i=1}^{r-1} U_r(x,r^{-e},i) \rightarrow \bigvee_{k=1}^{\ell} r^{-e} \in h_k(a,b,r^{-\ell\N})\right].
\]
By (2), we get that for all $x \in Y_{a,b}$,
\[
\forall r^{-e} \in r^{-\N} \left[\bigvee_{i=1}^{r-1} U_r(x,r^{-e},i) \rightarrow r^{-e+1} < b-\lambda(a,b)\right].
\]
In other words, if the digit corresponding to $r^{-e}$ in the base $r$ representation of element $x\in Y_{a,b}$ is positive, then $r^{-e}$ is smaller than $r^{-1}(b-\lambda(a,b))$. Thus, every element in $Y_{a,b}$ is smaller than $b-\lambda(a,b)$. Therefore, $\lambda(a,b) + Y_{a,b} \subseteq (a,b)$.\newline

\noindent Let $\nu \colon \Sigma_{r^\ell} \to \Sigma_r^\ell$ map $u\in \Sigma_{r^\ell}$ to the unique tuple $(v_0,\dots,v_{\ell-1})\in \Sigma_r$ such that
\[
u = \sum_{i=0}^{\ell-1} v_i r^{i}.
\]
Let $g_0 \colon Z \times D_{r^\ell}\to Y_{a,b}$ be given by
\[
\left(a,b,\sum_{d=1}^{n} u_d r^{-\ell d}\right) \mapsto \sum_{d=1}^{n} \sum_{i=0}^{\ell-1} v_{d,i} h_i(a,b,r^{-\ell d}),
\]
where $\nu(u_d) = (v_{d,0},\dots,v_{d,\ell-1})$ for $d\in \N_{>0}$. Since $W_{r^{\ell}}$ is definable in $(\R,<,+,W_r)$, so is $g_0$. By (I) and (II) and the uniqueness of finite base $r^\ell$ expansions, the function $g_0(a,b,-)$ is injective for fixed $(a,b)\in Z$. Observe that for all $a,b\in Z$, $\lambda(a,b) + g_0(a,b,D_{r^{\ell}})\subseteq (a,b)$.\newline

\noindent By Lemma \ref{lem:exp}(i), $D_r = D_{r^{\ell}}$. Let $g\colon \R^3 \times D_r \to D_r$ map $(c,a,b,d)$ to $0$ if $(a,b) \notin Z$, and otherwise to the $\prec_r$-minimal $e \in (D_r)_{\preceq d}$ such that $|c-(\lambda(a,b)+g_0(a,b,e))|$ is minimal. We can deduce from the injectivity of $g_0(a,b,-)$ that the ordered set $(D_r,\prec_r)$ together with the function $g$ satisfies the assumption of Fact \ref{fact:htplus}.
\end{proof}

\noindent Proposition \ref{prop:infinitepreimage} is essentially a result of Thomas \cite[Theorem 1]{Thomas1975}, which itself is a slight generalization of a classical result of Elgot and Rabin \cite[Theorem 1]{ElgotRabin}. Therefore, there is a more direct proof of Proposition \ref{prop:infinitepreimage} that invokes \cite[Theorem 1]{Thomas1975} instead of Fact \ref{fact:htplus}. However, the fact that Proposition \ref{prop:infinitepreimage} follows directly from Fact \ref{fact:htplus} should be of independent interest. Among other things, this means that Fact \ref{fact:htplus} can be thought of as a generalization of \cite[Theorem 1]{Thomas1975}.\newline

\noindent We now use Proposition \ref{prop:infinitepreimage} to deduce an analogue of a theorem of Villemaire (see B{\`e}s \cite[Theorem 4.2]{Bes-Undecidable}). The main argument is taken from the proof of \cite[Theorem 2]{ElgotRabin}.

\begin{cor}\label{cor:infinitepreimage} Let $g\colon r^{-\N}  \to r^{-\N}$ and $\ell\in \N_{>0}$ be such that
\begin{itemize}
\item[(i)] $g$ is strictly increasing,
\item[(ii)] for every $m\in \N$ there is $d\in \N$ such that $m\leq d \leq m+\ell$ and
\[
g(r^{-(d+1)})<r^{-1}g(r^{-d}).
\]
\end{itemize}
Then $(\R,<,+,W_r,g)$ defines every compact set.
\end{cor}
\begin{proof}
Let $B$ be the set of $r^{-d}\in r^{-\N}$ such that $g(r^{-d}) > r g(r^{-(d+1)})$.
Define $h_1\colon r^{-\N} \to r^{-\N}$ by
\[
h_1(r^{-d}) = \left\{
            \begin{array}{ll}
              r^{-e}, & \hbox{$r^{-d}=g^m(g(r^{-e})/r)$ for some $r^{-e} \in B$ and $m \in \N$,} \\
              1, & \hbox{otherwise.}
            \end{array}
          \right.
\]
It follows from the definition of $B$ that $h_1$ is well-defined. Observe that for all $x\in r^{-\N_{>0}}$, the set $h^{-1}(x)$ is infinite if and only if $x \in B$. We now show that $h_1$ is definable in $(\R,<,+,W_r,g)$.
Consider the set $X$ of all pairs $(r^{-d},r^{-e}) \in r^{-\N} \times r^{-\N}$ such that
\begin{align*}
\forall x \in [0,1) \Big( U_r(x,r^{-d},1) &\wedge \forall z \in r^{-\N} (U_r(x,g(z),1) \rightarrow U_r(x,z,1))\Big)\\
&\rightarrow U_r(x,g(r^{-e})/r,1).
\end{align*}
It is clear that $X$ is definable in $(\R,<,+,W_r,g)$. It can now be seen that the graph of $h_1$ is the union of $X$ with
\[
\{ (r^{-d},1)  \ : \ d\in \N, (r^{-d},r^{-e})\notin X \hbox{ for all } e\in \N\}.
\]
Let $h_2 \colon \R_{>0} \to r^{-\ell\N}$ map $x$ to $\max ((-\infty,x] \cap r^{-\ell\N})$. By (ii), we have $h_2(B) = r^{-\ell\N}$. Set $f:=h_2 \circ h_1$. Since $f^{-1}(x)$ is infinite for $x\in r^{-\ell\N}$, the structure $(\R,<,+,W_r,f)$ defines every compact set by Proposition \ref{prop:infinitepreimage}.
\end{proof}

\section{Proof of Theorem C}

Let $r,s\in\N_{\geq 2}$ be such that $\log_r(s) \notin \Q$. We will now show that $(\R,<,+,W_r,W_s)$ defines every compact set. As before, $D_r$ denotes the set of numbers in $[0,1)$ that admit a finite base $r$ expansion. For $t \in \N_{\geq 2}$, let $\supp(t)$ be the set of prime factors of $t$.

\subsection*{Case I: $\supp(r)\cap \supp(s)=\emptyset$} By Corollary \ref{cor:bool}, $(D_r-D_r)\cap (D_s-D_s)=\{0\}$. Therefore, $(\R,<,+,W_r,W_s)$ defines every compact set by Proposition \ref{lem:diff}.

\subsection*{Case II: $\supp(s)\subseteq \supp(r)$.} Let $m\leq n$ and write $r=p_1^{\alpha_1}\cdots p_n^{\alpha_n}$ and $s=p_1^{\beta_1}\cdots p_m^{\beta_m}$ for the prime factorizations of $r$ and $s$. Since $(\R,<,+,W_r)=(\R,<,+,W_{r^{\ell}})$ for every $\ell\in \N_{>1}$, we can assume that $1=\alpha_1/\beta_1 \leq \alpha_2/\beta_2\leq \cdots \alpha_m/\beta_m$. Since $r\neq s$, we also have $\log_s(r)>1$.\newline

\noindent Let $f\colon r^{-\N} \to r^{-\N}$ map $r^{-d}$ to $r^{-\lceil \log_s(r) d\rceil}$.

\begin{lem}\label{lem:rdivs}  The function $f$ is definable in $(\R,<,+,W_r,s^{-\N})$.
\end{lem}
\begin{proof}
It follows easily from Lemma \ref{lem:tau} and $\alpha_1=\beta_1$ that $\tau_r(s^{-d})=r^{-d}$ for every $d\in \N$. Let $\theta \colon \R_{>0} \to s^{-\N}$ map $x$ to $\max ((0,x] \cap s^{-\N})$. Observe that $\theta$ is definable in $(\R,<,+,s^{-\N})$. Since $s^{-\lfloor \log_s(r) d\rfloor} >  r^{-d} > s^{-\lceil \log_s(r) d\rceil}$, we have that $\theta(r^{-d}) = s^{-\lceil \log_s(r) d\rceil}$. Thus, $f= \theta\circ \tau_r$.
\end{proof}

\begin{prop}\label{prop:rdivs} The structure $(\R,<,+,W_r,s^{-\N})$ defines every compact set.
\end{prop}
\begin{proof}
Let $\ell \in \N$. Then,
\begin{align*}
f(r^{-(d+\ell)}) &= r^{-\lceil  \log_s(r)(d+\ell) \rceil}  \leq  r^{-(\lceil  \log_s(r)d\rceil + \lceil  \ell\log_s(r)\rceil - 1) } \\
&= f(r^{-d}) r^{- \lceil  \ell\log_s(r)\rceil - 1}.
\end{align*}
As $\log_s(r)>1$, we have  $rf(r^{-(d+1)})\leq f(r^{-d})$ and $\lceil (\ell+1)\log_s(r)\rceil >\ell + 1$ for sufficiently large $\ell \in\N$. Thus, there is $k \in \N$ such that $f(r^{-(d+k)})<r^{-k}f(r^{-d})$ for all $d\in \N$. Therefore, $f$ satisfies the assumption of Corollary \ref{cor:infinitepreimage}, so $(\R,<,+,W_r,f)$ defines every compact set. By Lemma \ref{lem:rdivs}, $(\R,<,+,W_r,s^{-\N})$ defines every compact set.
\end{proof}

\subsection*{Case III: Otherwise} Write $r=p_1^{\alpha_1}\cdots p_m^{\alpha_m}$ and $s=q_1^{\beta_1}\cdots q_n^{\beta_n}$ for the prime factorizations of $r$ and $s$. Let $t := \gcd(r,s)$, and without loss of generality let $t=p_1^{\gamma_1}\dots p_k^{\gamma_k}$ be the prime factorization of $t$.  By Cases I and II, we may assume that $0< k<m$. Let $u:=p_1^{\alpha}\cdots p_k^{\alpha_k}$ and
let $v:=p_{k+1}^{\alpha_{k+1}}\dots p_m^{\alpha_m}$.

\begin{lem}\label{lem:commonp} Let $d\in \N$. Then $u^{-d}$ is the $<$-smallest element $x$ in $D_u\setminus\{0\}$ such that $\tau_r(x)=r^{-d}$.
\end{lem}

\begin{proof}
It follows directly from Lemma \ref{lem:tau} that $\tau_r(u^{-d})=r^{-d}$. Let $w\in D_u\setminus\{0\}$ be such that $\tau_r(w)=r^{-d}$. By Lemma \ref{lem:exp}(ii), we have $wr^d\in\N_{>0}$ with $r\nmid wr^d$. As $r=uv$ and $w\in D_u$, it follows that $v^d\mid wr^d$. Thus $wu^d\in\N_{>0}$, so that $w\geq u^{-d}$.
\end{proof}


\noindent  By Lemma \ref{lem:bool}(iii), $D_u=D_t$. By Lemma \ref{lem:bool}(i), $D_t$ is definable in $(\R,<,+,W_r,W_s)$, and thus so is $D_u$. Combining this with Lemma \ref{lem:commonp}, we get that $(\R,<,+,W_r,W_s)$ defines $u^{-\N}$. By Proposition \ref{prop:rdivs}, $(\R,<,+,W_r,u^{-\N})$ defines every compact set.

  \bibliographystyle{plain}
  \bibliography{hieronymi}

\end{document}